\documentclass{article}
\usepackage{amsmath}
\usepackage{amsfonts}
\usepackage{amssymb}
\usepackage{amsthm}
\usepackage{mathtools}
\usepackage[numbers]{natbib}
\begin{document}
\title{Coleman automorphisms of finite groups and their minimal normal subgroups}
\author{Arne Van Antwerpen}
\date{}
\maketitle
\begin{abstract}In this paper, we show that all Coleman automorphisms of a finite group with self-central minimal non-trivial characteristic subgroup are inner; therefore the normalizer property holds for these groups. Using our methods we show that the holomorph and wreath product of finite simple groups, among others, have no non-inner Coleman automorphisms. As a further application of our theorems, we provide partial answers to questions raised by M. Hertweck and W. Kimmerle. Furthermore, we characterize the Coleman automorphisms of extensions of a finite nilpotent group by a cyclic $p$-group. Lastly, we note that class-preserving automorphisms of $2$-power order of some nilpotent-by-nilpotent groups are inner, extending a result by J. Hai and J. Ge. \end{abstract}

\section*{Introduction}
Let $G$ be a finite group and let $\mathbb{Z}G$ denote the integral group ring of $G$. Denote by $Z(U(\mathbb{Z}G))$ the center of the unit group $U(\mathbb{Z}G)$ of $\mathbb{Z}G$ and by $N_{U( \mathbb{Z}G)}(G)$ the normalizer of $G$ in $U(\mathbb{Z}G)$. The well known normalizer problem asks whether $N_{U(\mathbb{Z}G)}(G) = G Z(U(\mathbb{Z}G))$. This problem is posed as question $43$ in S.K. Sehgal \cite{Sehgalbook}. The first results for $p$-groups, nilpotent groups and groups with normal Sylow $2$-subgroups are due (respectively) to D. Coleman \cite{coleman1963}, A. Saksonov \cite{saksonov}, S. Jackowski and Z. Marciniak \cite{Jackowski1987}. Further work by, among others, S. Sehgal, M. Parmenter, Y. Li \cite{Blackburn}, E. Jespers and M. Hertweck \cite{blackburnJespers}, M. Hertweck \cite{Hertwecklocalanalysis,Hertweckclasspreservingcoleman,classpreservingHertweck} and M. Hertweck and W. Kimmerle \cite{HertweckKimmerle} enlarged the class of groups with a positive answer greatly. M. Mazur \cite{Mazur1995,Mazurinfinitegroups} showed that the question is closely related to the Isomorphism problem. It poses the question whether an isomorphism between two integral group rings $\mathbb{Z}G \cong \mathbb{Z}H$ implies that the groups are isomorphic $G \cong H$. In 2001 M. Hertweck \cite{Counterexample} discovered a counterexample to the isomorphism problem by first constructing a counterexample to the normalizer problem. All counterexamples to the normalizer problem, known to date, are obtained using M. Hertweck's construction \cite{Counterexample}. Hence, it opens the question whether there are still large classes of groups which have a positive answer to the normalizer problem (we say they satisfy the normalizer property). S. Jackowski and Z. Marciniak \cite{Jackowski1987} realized, by Coleman's result, that the normalizer property is closely related to automorphisms of the underlying finite group $G$.  Later, M. Hertweck and W. Kimmerle \cite{HertweckKimmerle} initiated the study of Coleman automorphisms, which are a generalization of the automorphisms obtained in the study of the normalizer property \cite{Jackowski1987}. Using the alternative characterization of the normalizer problem from \cite{Jackowski1987}, it is obvious that finite groups with no non-inner Coleman automorphisms have the normalizer property.

In the first section we give the necessary background on the automorphisms in play.

In the second section we show that the Coleman automorphisms of groups with a self-centralizing minimal non-trivial characteristic subgroup are trivial. Using the same techniques, we show that the holomorph of finite simple groups has no non-inner Coleman automorphisms. Also, we partially answer questions $2$ and $3$ posed by M. Hertweck and W. Kimmerle in \cite{HertweckKimmerle}. Moreover, we show that if question $2$ is answered positively, then so is question $3$. Furthermore, the theorems proven in this section show that certain semi-direct products with simple groups, $p$-groups or symmetric groups as normal base group have no non-inner Coleman automorphisms. In particular, the recent results of J. Hai and J. Guo in \cite{haipermutational} are easy consequences of our theorems.

In the third section, Coleman automorphisms of extensions of nilpotent groups by cyclic $p$-groups are characterized. As a nice consequence, this provides an easier and different proof of E.C. Dade's result (in \cite{DadeLocallytrivial}) that any finite abelian group can be realized as the group of outer Coleman automorphisms of some metabelian group $G$. Furthermore, the recent results of Y. Li and Z. Li in \cite{generalizeddihedral} are straightforward applications of our characterization. It also shows that the theorems of the previous sections can not be naively adapted to the case of a semi-direct product with a nilpotent group as a normal base group.

In the last section, we note that a recent result of J. Hai and S. Ge ( in \cite{nilpotentbycyclic}) can easily be generalized. It is shown that the class-preserving Coleman automorphisms of $p$-power order of certain nilpotent-by-cyclic groups are inner. As a Corollary we then prove that the class-preserving Coleman automorphisms of $p$-power order of certain nilpotent-by-nilpotent groups are inner.
\section{Preliminary results}
In this section we give the definition of the investigated automorphisms. We also include crucial known results on these automorphisms. By $\textup{Aut}(G)$ (respectively $\textup{Inn}(G)$) we denote the automorphisms (respectively inner automorphisms) of a group $G$. By $\textup{conj}(g)$ we denote the inner conjugation $x \mapsto g^{-1} xg$ on $G$. The group of outer automorphisms is $\textup{Out}(G) = \textup{Aut}(G)/\textup{Inn}(G)$. The set consisting of the prime divisors of the order of a finite group $G$ is denoted by $\pi(G)$. The cyclic group of order $n$ will be denoted $C_n$ for a positive integer $n$. The commutator of two elements $x,y \in G$ is denoted $[x,y]$.

We begin with the definition of Coleman automorphism, which was introduced by M. Hertweck and W. Kimmerle in \cite{HertweckKimmerle}.
\newtheorem{Definition}[subsection]{Definition} \begin{Definition}\label{defcoleman} Let $G$ be a finite group and $\sigma \in \textup{Aut}(G)$. If for any prime $p$ dividing the order of $G$ and any Sylow $p$-subgroup $P$ of $G$, there exists a $g \in G$ such that $\sigma|_P = \textup{conj}(g)|_P$, then $\sigma$ is said to be a Coleman automorphism.
\end{Definition}
\vspace{0.1cm}
The group of Coleman automorphisms is denoted $\textup{Aut}_{col}(G)$ and its image in $\textup{Out}(G)$ is denoted $\textup{Out}_{col}(G)= \textup{Aut}_{col}(G)/\textup{Inn}(G)$.
\begin{Definition}\label{defclasspres} Let $G$ be a finite group and $\sigma \in \textup{Aut}(G)$. If the conjugacy classes of $G$ are invariant under $\sigma$, i.e. for any $g \in G$ there exists an $x \in G$ such that $\sigma(g) = x^{-1}gx$, then $\sigma$ is called a class-preserving automorphism.
\end{Definition}
\vspace{0.1cm}
The group of class-preserving automorphisms is denoted $\textup{Aut}_{c}(G)$ and its image in $\textup{Out}(G)$ is denoted $\textup{Out}_{c}(G)= \textup{Aut}_{c}(G)/\textup{Inn}(G)$. The following proposition shows that the prime divisors of the orders of the groups of class-preserving automorphisms and Coleman automorphisms are restricted to the prime divisors of $|G|$. This is well known for the class-preserving automorphisms, a proof can be found in \cite[Proposition 2.4]{Jackowski1987}. A proof for the Coleman automorphisms may be found in \cite[Proposition 1]{HertweckKimmerle}.
\newtheorem{Theorem}[subsection]{Theorem} \begin{Theorem}\label{orderclasspreserving} Let $G$ be a finite group. The prime divisors of $|\textup{Aut}_{col}(G)|$ and $|\textup{Aut}_{c}(G)|$ are also prime divisors of $|G|$.
\end{Theorem}
The following well-known theorem shows that Coleman automorphisms behave well with respect to direct products.
\newtheorem{Proposition}[subsection]{Proposition} \begin{Proposition}\label{directprodcoleman} Let $G$ and $H$ be finite groups. Then $\textup{Aut}_{col}(G \times H) \cong \textup{Aut}_{col}(G) \times \textup{Aut}_{col}(H)$ and $\textup{Out}_{col}(G \times H) \cong \textup{Out}_{col}(G) \times \textup{Out}_{col}(H)$.
\end{Proposition}
The following lemma shows that Coleman automorphisms are ''close'' to being class-preserving (this result appeared implicitly in \cite{HertweckKimmerle}).
\newtheorem{Lemma}[subsection]{Lemma}\begin{Lemma}\label{colemannormalinvariant} Let $N$ be a normal subgroup of a finite group $G$. If $\sigma \in \textup{Aut}_{col}(G)$, then $\sigma|_{N} \in \textup{Aut}(N)$. \end{Lemma}
\begin{proof}
Let $x \in N$ and $\sigma\in \textup{Aut}_{col}(G)$. Let $p_1, ... , p_k$ be the distinct prime divisors of the order of $x$ and write $x = x_{p_1} ... x_{p_k}$, where each $x_{p_i}$ is the $p_i$-part of $x$. As $\sigma \in \textup{Aut}_{col}(G)$, there exist $g_i \in G$ such that $\sigma(x_{p_i}) = g^{-1}_i x_{p_i} g_i$. Thus, $\sigma(x) = g_1^{-1} x_{p_1} g_1 ... g_k^{-1} x_{p_k} g_k$ and therefore $\sigma(x) \in N$.
\end{proof}
A proof of the following lemma can be found in \cite[Lemma 2]{classpreservingHertweck}.
\begin{Lemma}\label{lemmaopzn} Let $p$ be a prime number and $\alpha \in \textup{Aut}(G)$ of $p$-power order. Assume that there exists a normal subgroup $N$ of $G$, such that $\alpha|_N = \textup{id}_N$ and $\alpha$ induces identity on $G/N$. Then $\alpha$ induces identity on $G/O_p(Z(N))$. Moreover, if $\alpha$ also fixes a Sylow $p$-subgroup of $G$ elementwise, then $\alpha$ is conjugation by an element of $O_p(Z(N))$.
\end{Lemma}
\vspace{0.1cm}
The following result was proven by M. Hertweck and W. Kimmerle in \cite{HertweckKimmerle}.
\begin{Theorem}\label{p'groupserfelijk} Let $N$ be a normal subgroup of a finite group $G$ and $p$ a prime number which does not divide the order of $G/N$. Then the following properties hold.
\begin{enumerate}
\item If $\sigma \in \textup{Aut}(G)$ is a class-preserving or a Coleman automorphism of $G$ of $p$-power order, then $\sigma$ induces a class-preserving or a Coleman automorphism of $N$ respectively.
\item If $\textup{Out}_c(N)$ or $\textup{Out}_{col}(N)$ is a $p'$-group, then so is $\textup{Out}_c(G)$ or $\textup{Out}_{col}(G)$ respectively.
\end{enumerate}
\end{Theorem}
The second part of the theorem can be extended as follows. The proof is similar and is included for completeness' sake.
\begin{Theorem}\label{Corollaryp'groupserfelijk}
Let $N$ be a normal subgroup of a finite group $G$ and $p$ a prime number which does not divide the order of $G/N$. If $\textup{Out}_c(N) \cap \textup{Out}_{col}(N)$ is a $p'$-group, then so is $\textup{Out}_c(G) \cap \textup{Out}_{col}(G)$.
\end{Theorem}
\begin{proof}
Assume $\textup{Out}_c(N) \cap \textup{Out}_{col}(N)$ is a $p'$-group. Let $\sigma$ be a class-preserving Coleman automorphism of $G$ of $p$-power order. By Theorem \ref{p'groupserfelijk}, $\sigma$ induces a class-preserving Coleman automorphism of $N$ of $p$-power order. Hence, there exists $x \in N$ such that $\sigma|_N = \textup{conj}(x)|_N$. Put $\sigma_1 = \sigma \circ \textup{conj}(x^{-1})$. Let $m$ be the $p'$-part of the order of $\sigma_1$ and put $\sigma_2 = \sigma_1^m$. Then $\sigma_2$ differs from $\sigma$ only by an inner automorphism and a $p'$-power. Thus it is sufficient to prove that $\sigma_2$ is inner. Hence, we may assume that $\sigma$ is a class-preserving Coleman automorphism of $p$-power order such that $\sigma|_N = \textup{id}_N$. So, by Theorem \ref{orderclasspreserving}, $\sigma$ induces identity on $G/N$. As the Sylow $p$-subgroups of $G$ are contained in $N$, it follows, by Lemma \ref{lemmaopzn}, that $\sigma$ is an inner automorphism of $G$.
\end{proof}
\vspace{0.1cm}
An analogue to Coleman automorphims was introduced by F. Gross in \cite{grosspcentral}.
\begin{Definition}\label{pcentral} Let $G$ be a finite group and let $\sigma \in \textup{Aut}(G)$ be an automorphism. If for some prime divisor $p$ of $|G|$ the automorphism $\sigma$ centralizes a Sylow $p$-subgroup $P$ of $G$, i.e. $\sigma|_P = \textup{id}|_P$, then $\sigma$ is said to be a $p$-central automorphism of $G$.
\end{Definition}
\vspace{0.1cm}
Note that every Coleman automorphism of a group $G$ is, upto an inner automorphism, a $p$-central automorphism for any prime $p \in \pi(G)$. In \cite[Theorem B]{grosspcentral} F. Gross showed the following useful theorem (note that the formulation differs somewhat from the original).
\begin{Theorem} \label{Grosspcentralallp} Let $G$ be a finite non-abelian simple group and let $p$ be an odd prime dividing the order of $G$. Then every $p$-central automorphism of $G$ of $p$-power order is inner. \end{Theorem}
\vspace{0.1cm}
M. Hertweck and W. Kimmerle continued in \cite{HertweckKimmerle} the study of these $p$-central automorphisms for simple groups and they proved the following interesting result.
\begin{Theorem} \label{Simplegroupspcentral} For any finite simple group $G$, there is a prime $p$ dividing $|G|$ such that $p$-central automorphisms of $G$ are inner automorphisms.
\end{Theorem}

\section{Groups with a self-centralizing characteristic subgroup}
In this section Coleman automorphisms of finite groups with a self-centralizing minimal non-trivial characteristic subgroup are investigated. Our results on these groups will partially answer several questions posed by M. Hertweck and W. Kimmerle in \cite{HertweckKimmerle}. Furthermore, these results will imply several results showing that the Coleman automorphisms of certain semi-direct products, in particular wreath products, have no non-inner Coleman automorphisms. The results in the recent paper \cite{haipermutational}, on  wreath products of symmetric groups, are straightforward applications of our results. Moreover, it is shown that a result of F. Gross, Corollary $2.4$ in \cite{grosspcentral}, can be slightly generalized.
\begin{Theorem}\label{selfcentralcharacteristic} Let $G$ be a normal subgroup of a finite group $K$. Let $N$ be a minimal non-trivial characteristic subgroup of $G$. If $C_K(N) \subseteq N$, then every Coleman automorphism of $K$ is inner. In particular, the normalizer problem holds for these groups. \end{Theorem}
\begin{proof}
Note that a minimal non-trivial characteristic subgroup of $G$ is necessarily characteristically simple. Hence, $N$ is the direct product $\Pi_{i=1}^n S_i$ of finite simple groups $S_i$ each isomorphic with a fixed simple group, say $S$. Abusing notation, we identify each $S_i$ with $S$. By Theorem \ref{Simplegroupspcentral}, there exists a prime number $q \in \pi\left(S\right)$ such that all $q$-central automorphisms of $S$ are inner. Let $\sigma \in \textup{Aut}(K)$ be a Coleman automorphism of $p$-power order. First, we show that we may assume this automorphism to be $q$-central. Let $Q$ be a Sylow $q$-subgroup of $S$. Then there exists a Sylow $q$-subgroup $P$ of $K$ such that $\Pi_{i=1}^n Q_i \subseteq P$, where $Q_i$ denotes the image of $Q$ under the canonical isomorphism from $S$ to $S_i$. Put $\rho_1= \sigma \circ \textup{conj}(h^{-1})$. As $\sigma|_P = \textup{conj}(h)|_P$ for some $h \in K$, it follows that $\rho_1$ is a Coleman automorphism which restricts to the identity on $P$. Let $m$ denote the $p'$-part of the order of $\rho_1$. Put $\rho_2= \rho_1^m$. Then $\rho_2$ is a Coleman automorphism of $p$-power order and $\rho_2|_P = \textup{id}_P$. As $\rho_2\sigma^{-m} \in \textup{Inn}(G)$ and $p \nmid m$, we may replace $\sigma$ by $\rho_2$. Hence, without loss of generality, we may assume that $$\sigma|_P = \textup{id}|_P.$$

So, indeed, $\sigma$ is $q$-central. In particular, for any $1 \leq i \leq n$, $\sigma|_{Q_i} = \textup{id}_{Q_i}$.  Note that the normal subgroup generated by $Q_i$ in $S_i$ is precisely $S_i$. By Lemma~\ref{colemannormalinvariant}, $\sigma|_{N} \in \textup{Aut}(N)$. Hence, for any $1 \leq i \leq n, x \in S_i$ and $g \in Q_i$, $$ \sigma(x^{-1} gx) = \sigma(x)^{-1} g \sigma(x) \in S_i.$$

As $\sigma(x) \in N$, this shows, for any $ 1 \leq i \leq n$, that $\sigma|_{S_i} \in \textup{Aut}(S_i)$. Thus, each $\sigma|_{S_i}$ is a $q$-central automorphism. Hence, there exists, for any $ 1 \leq i \leq n$, a $g_i \in S_i$ such that $$ \sigma|_{S_i} = \textup{conj}(g_i)|_{S_i}.$$
Thus, $\sigma|_N = \textup{conj}(\Pi_{i=1}^ng_i)$.

The proof now continues for two mutually exclusive cases.
\vspace{0.5cm}

First we consider the case that $S$ is non-abelian. Put $\sigma_1 = \sigma \circ \textup{conj}\left(\Pi_{i=1}^ng_i^{-1}\right)$. This is a Coleman automorphism of $K$ and $\sigma_1|_{N} = \textup{id}_N$. Let $m$ denote the $p'$-part of the order of $\sigma_1$. Put $\sigma_2= \sigma_1^m$, a Coleman automorphism of $p$-power order such that $\sigma_2|_N = \textup{id}_N$. As $\sigma_2 \sigma^{-m} \in \textup{Inn}(G)$ and $p\nmid m$, we may replace $\sigma$ by $\sigma_2 \sigma^{-m}$. Hence, without loss of generality, we may assume, except that $\sigma$ may not be $q$-central anymore, that $\sigma|_N = \textup{id}_N$. As $N$ is characteristic in $G$ it follows, for any $x \in K, n \in N$, that $$ x^{-1} nx = \sigma(x^{-1}nx) = \sigma(x)^{-1} n \sigma(x).$$
Hence, $$\sigma(x) x^{-1} \in C_K(N).$$
Since $S$ is non-abelian, it follows that $Z(N) = 1$.
This shows that $\sigma = \textup{id}_K$, as desired.
\vspace{0.5cm}

Second, we consider the case that $S$ is abelian, and thus $S = C_q$, a cyclic group of order $q$. As $S$ and $N$ are abelian and $\sigma|_N$ is inner, it follows that $\sigma|_N = \textup{id}_N$. Moreover, as $N$ is characteristic, it follows, for any $x \in K, n \in N$, that $$ x^{-1} n x = \sigma(x^{-1}nx) = \sigma(x)^{-1} n \sigma(x). $$
Hence, $\sigma(x)x^{-1} \in C_K(N) \subseteq N$. This shows that $$ \sigma|_N = \textup{id}_N \qquad \textnormal{ and } \qquad \sigma|_{G/N} = \textup{id}_{G/N}. $$

If $q=p$, it follows, by Lemma \ref{lemmaopzn} and $q$-centrality of $\sigma$, that $ \sigma$ is inner.

Assume $q \neq p$. Then, for any $x \in K$, there exists an $n \in N$ such that $ \sigma(x) = xn$. Let $p^r$ be the order of $\sigma$, then $$ x = \sigma^{p^r}(x) = xn^{p^r}.$$ As $N$ is a $q$-group, this implies that $n = e$ and hence $\sigma = \textup{id}$. So we have shown that any $p$-power order $\sigma \in \textup{Aut}_{col}(K)$ is inner. Hence, the result follows.
\end{proof}

Similar arguments are used to prove the following theorem. This theorem is similar to F. Gross' Corollary $2.4$ in \cite{grosspcentral}.
\begin{Theorem}\label{selfcentralpgroups} Let $P$ be a normal $p$-subgroup of a finite group $G$. If $C_{G}(P) \subseteq P$, then $G$ has no non-inner $p$-central automorphisms. In particular, $G$ has no non-inner Coleman automorphisms. \end{Theorem}
\begin{proof}
Let $q$ be a prime number. Let $\sigma$ be a $p$-central automorphism of $q$-power order. Hence, because $P$ is normal by the assumptions, $\sigma|_P = \textup{id}_P$. Therefore, for any $g \in G$ and $x \in P$, $$ g^{-1}xg = \sigma(g^{-1} xg) = \sigma(g)^{-1} x \sigma(g).$$
Thus, because of the assumption, $\sigma(g) g^{-1} \in C_G(P) \subseteq P$.

Again we consider two cases.

First, suppose $q \neq p$. Then, by the above, for any $g \in G$, there exists an $h \in P$ such that $$ \sigma(g) = hg.$$
If $q^n$ is the order of $\sigma$, then this shows that $$ g = \sigma^{q^n}(g) = h^{q^n} g.$$ As $h$ is a $p$-element, it follows that $h = e$. So we have shown that $\sigma = \textup{id}_G$.

Second, assume $q = p$. Because $\sigma|_P = \textup{id}_P$, $\sigma_{G/P} = \textup{id}_{G/P}$ and since $\sigma$ fixes a Sylow $p$-subgroup, it follows, by Lemma~\ref{lemmaopzn}, that $\sigma$ is an inner automorphism of $G$.

In both cases we have shown that any $p$-central automorphism of $G$ of prime power order is inner. As any Coleman automorphism is $p$-central upto an inner automorphism, the result follows.
\end{proof}
As immediate consequence of Theorem \ref{selfcentralpgroups} we get the following.
\newtheorem{Corollary}[subsection]{Corollary}\begin{Corollary}\label{pgroups} Let $P$ be a finite $p$-group and let $H$ be a finite group. Let $G = P \rtimes H$ be the associated semidirect product. If $C_{G}(P) \subseteq P$, then $G$ has no non-inner Coleman automorphisms. \end{Corollary}
Recall that the wreath product $G \wr H$ of two finite groups $G,H$ is defined as the semidirect product $\Pi_{h \in H} S \rtimes H$, where $H$ acts by left multiplication on the indices. An immediate consequence of Theorem \ref{selfcentralcharacteristic} is the following.
\begin{Corollary}\label{simple} Let $S$ be a finite simple group, $I$ a finite set of indices, and $H$ any finite group. If $G=\Pi_{i \in I}S \rtimes H$ is a group such that $C_G(\Pi_{i \in I} S) \subseteq Z( \Pi_{i \in I} S)$, then $G$ has no non-inner Coleman automorphism, and thus, the normalizer problem holds for $G$. In particular, the wreath product $ S \wr H$ has no non-inner Coleman automorphisms. \end{Corollary}
As straightforward applications of our theorems, we mention a recent result by J. Hai and S. Ge.
\begin{Corollary} Let $k,n$ be positive integers. Then $S_k \wr S_n$ has no non-inner Coleman automorphisms. In particular, $S_k \wr S_n$ satisfies the normalizer property. \end{Corollary}
Another application of the reasoning in our theorems is in the holomorph of groups. Recall that, for a finite group $G$, the holomorph of $G$ is defined as $\textup{Hol}(G) = G \rtimes \textup{Aut}(G)$, where the action of $\textup{Aut}(G)$ is by evaluating the automorphisms.
\begin{Theorem} Let $S$ be a non-abelian finite simple group. Then $\textup{Hol}(S)$ has no non-inner Coleman automorphisms. \end{Theorem}
\begin{proof}
Let $p$ be a prime number. Let $\sigma \in \textup{Aut}_{col}(\textup{Hol}(S))$ of $p$-power order. By Theorem \ref{Simplegroupspcentral}, there exists a prime number $q$ such that $q$-central automorphisms of $S$ are inner. We may assume $\sigma$ to be $q$-central of $p$-power order, by modifying $\sigma$ by an inner automorphism as in the proof of Theorem \ref{selfcentralcharacteristic} and taking a suitable power. By Lemma \ref{colemannormalinvariant}, $\sigma|_{S \rtimes S} \in \textup{Aut}(S \rtimes \textup{Inn}(S))$ and $\sigma|_S \in \textup{Aut}(S)$. Hence, $\sigma$ induces a $q$-central automorphism of $p$-power order on $S \rtimes \textup{Inn}(S)/S \cong \textup{Inn}(S) \cong S$. Because of Theorem \ref{Simplegroupspcentral}, $\sigma$ induces an inner automorphism $\textup{conj}(\textup{conj}(\overline{x}))$ on $S \rtimes \textup{Inn}(S)/ S$, where $x \in C_S(Q)$ is of $p$-power order for some Sylow $q$-subgroup $Q$ of $S$. Hence, $\sigma_1 = \sigma \circ \textup{conj}(\textup{conj}(x^{-1}))$ is still $q$-central, as it fixes $Q \rtimes \textup{conj}(Q)$ elementwise and induces the identity on $S \rtimes \textup{Inn}(S)/S$. Thus, by Theorem \ref{Simplegroupspcentral}, $\sigma_1|_S = \textup{conj}(s)$ for some $s \in C_S(Q)$. Hence, $\sigma_2 = \sigma_1 \circ \textup{conj}(s^{-1})$ is a $q$-central automorphism of $S \rtimes \textup{Inn}(S)$, $\sigma_2$ restricts to the identity on $S$ and $\sigma_2$ induces identity on $S \rtimes \textup{Inn}(S)/S$. Write $\sigma_2$ as $\sigma$ for convenience. Let $w, r \in S$ and $q \in S$, then \begin{align*} \sigma( w \textup{conj}(r))^{-1} q \sigma( w \textup{conj}(r)) &= \sigma ( \textup{conj}(r^{-1}) w^{-1} q w \textup{conj}(r)) \\ &= \textup{conj}(r^{-1}) w^{-1} q w\textup{conj}(r). \end{align*}
Hence, and because $\sigma|_{S \rtimes \textup{Inn}(S)/S} = \textup{Id}_{S \rtimes \textup{Inn}(S) / S}$, $$ w \textup{conj}(r)\sigma( w \textup{conj}(r))^{-1}  \in C_{S \rtimes \textup{Inn}(S)}(S) \cap S = 1.$$ Thus, $\sigma|_{S \rtimes \textup{Inn}(S)} = \textup{Id}_{S \rtimes \textup{Inn}(S)}$. A similar argument shows that, for $g \in \textup{Hol}(S)$, $$\sigma(g) g^{-1} \in C_{\textup{Hol}(S)}(S \rtimes \textup{Inn}(S)).$$ Clearly, $C_{\textup{Hol}(S)}(S \rtimes \textup{Inn}(S)) = 1$. Hence, $ \sigma = \textup{Id}_{\textup{Hol}(S)}$.
\end{proof}
A last and easy application of our theorems is an alternative proof and extension of a recent result of J. Hai, S. Ge and W. He \cite{Haiholomorphnilpotentsymmetric}.
\begin{Corollary} Let $N$ be a finite nilpotent group. Then $\textup{Hol}(N)$ has no non-inner Coleman automorphism. In particular, $\textup{Hol}(N)$ has no non-inner Coleman automorphisms. \end{Corollary}
\begin{proof}
Put $N = P_1 \times ... \times P_k$ in its Sylow decomposition. Then $\textup{Hol}(N) \cong \Pi_{i=1}^k \textup{Hol}(P_i)$. By Proposition \ref{directprodcoleman}, it is sufficient to show the result in the case that $N$ is a $p$-group for some prime number $p$. Clearly, $N \rtimes \textup{Inn}(N)$ is a self-central normal $p$-subgroup of $\textup{Hol}(N)$. Thus, by Theorem \ref{selfcentralpgroups}, $\textup{Hol}(N)$ has no non-inner Coleman automorphisms.
\end{proof}
The following lemma easily follows from the fact that the automorphism group of the direct product of characteristic subgroups is the direct product of the automorphism groups.
\begin{Lemma}\label{directprodautomorphisms} Let $S_1,...,S_n$ be non-isomorphic finite simple groups and $k_1, ..., k_n$ positive integers. Then $\textup{Aut}(S_1^{k_1}\times ... \times S_n^{k_n}) = \textup{Aut}(S_1^{k_1}) \times ... \textup{Aut}(S_n^{k_n})$. \end{Lemma}
Recall that the generalized Fitting subgroup $F^*(G)$ of a finite group $G$ is defined as $F^*(G) = F(G) E(G)$, where $F(G)$ is the Fitting subgroup (the maximal nilpotent normal subgroup) and $E(G)$ is the layer, i.e. the maximal semisimple normal subgroup. It is well known that $F^*(G)$ is self-centralizing and the elements of the layer and Fitting subgroup commute. For more in depth results, we refer to \cite[Section X.13]{Huppert}.
As an application of the main results of this section we give some partial answers to the following questions posed by M. Hertweck and W. Kimmerle in \cite{HertweckKimmerle}; where $p$ is a prime number and $G$ is a finite group.\\
\begin{enumerate}
\item Is $\textup{Out}_{col}(G)$ a $p'$-group if $G$ does not have $C_p$ as a chief factor?
\item Is $\textup{Out}_{col}(G) $ trivial if $O_{p'}(G)$ is trivial?
\item Is $\textup{Out}_{col}(G)$ trivial if $G$ has a unique minimal non-trivial normal subgroup?
\end{enumerate}
M. Hertweck and W. Kimmerle showed that all these questions have affirmative answers in case $G$ is $p$-constrained, for example if $G$ is solvable. Recall that every finite solvable group has abelian minimal non-trivial normal subgroups. Concerning the third question, as the minimal non-trivial normal subgroups are abelian, it follows that under the extra assumption that this minimal normal subgroup is abelian, a partial answer was given by M. Hertweck and W. Kimmerle. In the next theorem we show that if this minimal normal subgroup is non-abelian, the statement always has a positive answer.

\begin{Theorem}\label{nonabelianminimal} Let $G$ be a finite group with a unique minimal non-trivial normal subgroup $N$. If $N$ is non-abelian, then $\textup{Out}_{col}(G) = 1$.
\end{Theorem}
\begin{proof}
As $C_G(N)$ is a normal subgroup of $G$ and $N$ is non-abelian, the minimality assumption yields that $C_G(N) = 1$. Obviously, the assumptions yield that $N$ is a self-centralizing minimal non-trivial characteristic subgroup of $G$. Thus, the result follows from Theorem \ref{selfcentralcharacteristic}.
\end{proof}
Thus to answer question $3$, only groups with an abelian minimal non-trivial normal subgroup need to be checked. Moreover, these groups may be assumed to be not $p$-constrained. We will show now that a positive answer to the second question implies a positive answer to the third question.
\begin{Proposition}\label{question2impliesquestion3} If Question $2$ has a positive answer, then so does Question $3$. \end{Proposition}
\begin{proof}
By Theorem \ref{nonabelianminimal}, it is sufficient to consider finite groups $G$ with an abelian unique minimal non-trivial normal subgroup, say $N$. Then $N$ is a characteristically simple group. Hence, $N$ is an elementary abelian $p$-group for some prime $p$. Thus $N \subseteq O_p(G)$. As $N$ is the unique minimal normal subgroup, it follows that $O_{p'}(G) = 1$. So,the result follows.
\end{proof}
Denote by $O_p(G)$ the maximal normal $p$-subgroup of a group $G$. The following theorem serves to provide a partial answer to question $3$ under the additional assumption that $O_p(G) = 1$.
\begin{Theorem} \label{outcol2} Let $K$ be a finite group and $p$ be an odd prime number. Suppose $K$ has a normal subgroup $E$, which is the direct product of non-abelian simple groups, such that $C_K(E) \subseteq E$. If $p$ divides the order of every simple group in the decomposition of $E$, then $\textup{Out}_{col}(K)$ is a $p'$-group. \end{Theorem}
\begin{proof}
Let $\sigma$ be a Coleman automorphism of $p$-power order. Put $E = S_1^{k_1} \times ... \times S_n^{k_n}$ with $S_i \not\cong S_j$ if $i \neq j$. By Lemma \ref{colemannormalinvariant}, $ \sigma|_E \in \textup{Aut}(E)$. Let, for any $1 \leq i \leq n$, $Q_i$ be a Sylow $p$-subgroup of $S_i$. Then $Q=Q_1^{k_1} \times ... \times Q_n^{k_n}$ is a $p$-subgroup of $E$. By the same reasoning as in the proof of Theorem \ref{selfcentralcharacteristic}, we may assume, without loss of generality, that $\sigma|_Q = \textup{id}|_Q$. Each $S_i^{k_i}$ is a characteristic subgroup in $S_1^{k_1}\times ... \times S_n^{k_n}$, thus $\sigma|_{S_i^{k_i}} \in \textup{Aut}(S_i^{k_i})$. Let us identify the $l$'th copy of $S_i$ (denoted $S_i^{(l)}$) in $S_i^{k_i}$ with $S_i$. Then for any $x \in S_i$ and for any $q \in Q_i$, it follows that $$ \sigma(x^{-1}q x) = \sigma(x)^{-1} q \sigma(x).$$
Because of the assumption, each $Q_i$ is non-trivial and, clearly, the normal subgroup generated by $Q_i$ (in $E$) is $S_i$. As $\sigma(x) \in E$, this shows that $\sigma|_{S_i^{(l)}} \in \textup{Aut}(S_i^{(l)})$. Since $\sigma|_{S_i^{(l)}}$ is $p$-central and $p$ is odd, it follows by  Theorem \ref{Grosspcentralallp} that there exists $x_i^{(l)} \in S_i^{(l)}$ such that $$ \sigma|_{S_i^{(l)}} = \textup{conj}(x_i^{(l)}).$$ As this holds for any $l\leq k_i$ and any $i \leq n$, there exists $x \in E$ such that $$\sigma|_E = \textup{conj}(x).$$
Hence, as in previous proofs, modification by an inner automorphism and taking a suitable power, we may assume, without loss of generality, that $\sigma|_E = \textup{id}_E$, but here we possibly sacrifice $p$-centrality. Thus, for any $g \in E$ and $x \in K$, it follows that $$ x^{-1} gx = \sigma(x^{-1}gx) = \sigma(x)^{-1} g \sigma(x).$$
This shows, for any $ x \in K$, that $\sigma(x)x^{-1} \in C_K(E) \subseteq E.$ So $\sigma(x)x^{-1} = e$ and thus $\sigma = \textup{id}_K$.
\end{proof}
\begin{Theorem} \label{outcolspecial} Let $G$ be a finite group and $p$ an odd prime number. If $O_p(G) = O_{p'}(G)= 1$ and $p$ divides the order of every simple subgroup of $E$, then $\textup{Out}_{col}(G)$ is a $p'$-group. \end{Theorem}
\begin{proof}
Clearly, the assumptions implies that $F(G) = 1$. Hence, $F^*(G) = E(G)$. This shows that $C_G(E(G)) \subseteq E(G)$. Moreover, as $Z(E(G)) \subseteq F(G) = 1$, it follows that $E(G)$ is the direct product of non-abelian simple groups. The result then follows from Theorem \ref{outcol2}.
\end{proof}

\section{Coleman automorphisms of Nilpotent-by-cyclic groups}
In this section we obtain a characterization of the Coleman automorphisms of finite nilpotent-by-($p$-power cyclic) groups. Apart from giving concrete examples of groups with known non-trivial outer Coleman automorphisms, we obtain an easy second proof of E.C. Dade's result \cite{DadeLocallytrivial, DadeCorrection}, which shows that any finite abelian group can be realized as the group of outer Coleman automorphisms of some finite metabelian group. Furthermore, our results show that the results of the previous section are not to be expected for semi-direct products of a nilpotent group and an arbitrary finite group (with faithful action). To give some more insight in the matter, as the proof is quite technical, we first show the result for a semi-direct product of a finite abelian group by a cyclic group of prime power order. Our work is inspired by recent work of Z. Li and Y. Li \cite{generalizeddihedral}. They studied the group of Coleman automorphisms of a finite generalized dihedral group and obtained a concrete characterization of this group. First, we define the automorphisms, which will induce the only outer Coleman automorphisms.
\begin{Definition}
Let $ A * C_q$ be the free product of a finite abelian group $A$ and a cyclic group $C_q = \left< x \right>$ of order $q=p^r$, with $p$ a prime number. Write $A = A_{p_1} \times ... \times A_{p_n}$, its Sylow decomposition. Then, for any $1\leq i \leq n$ and non-negative integers $k_i$, define the group homomorphisms \begin{align*} \phi_{k_1...k_n}: A * C_q &\rightarrow A * C_q \end{align*} where, for  $g_i \in A_{p_i}$,,  \begin{align*} g_i \mapsto x^{-k_i} g_i x^{k_i} &\textnormal{ and }x \mapsto x
\end{align*}
\end{Definition}

It easily is verified that these maps induce Coleman automorphisms $\Phi_{k_1 ... k_{n-1}k_n}$ on all possible semi-direct products.

\begin{Theorem}\label{Theoremoutcol}
Let $G = A \rtimes C_q$, with notation as in the definition above. For any $1 \leq i \leq n$, let $r_i$ be the smallest positive integer such that $\textup{conj}(x^{r_i})|_{A_{p_i}} = \textup{id}|_{A_{p_i}}$. Assume that the Sylow subgroups are ordered such that $r_1 \leq ... \leq r_n$.  Then, $$ \textup{Aut}_{col}(G) = \textup{Inn}(G) \rtimes K, $$ with $$ K = \left\lbrace \Phi_{k_1 ... k_{n-1}0} \mid k_i \in \mathbb{Z}_{r_i} \right\rbrace.$$ Thus, $$ K \cong C_{r_1} \times ... \times C_{r_{n-1}}.$$
\end{Theorem}
\begin{proof}
Clearly, $K$ is a subgroup of $\textup{Aut}_{col}(G)$. First, we study when the automorphisms $\Phi_{k_1 ... k_{n}}$ and $\Phi_{j_1...j_n}$ determine the same outer automorphism of $G$. For this, let $j_i,k_i \in \mathbb{Z}_{r_i}$ and $g \in G$ be such that $$\Phi_{j_1...j_n} = \Phi_{k_1...k_n} \circ \textup{conj}(g).$$
Write $g = bx^l$, with $b \in A$ and $l$ a non-negative number. Then, for any $1 \leq i \leq n$ and $a_i \in A_{p_i}$, $$ x^{-j_i} a_i x^{j_i} = x^{-k_i - l}b^{-1} a_i b x^{l+k_i},$$
Equivalently, $$ x^{k_i - j_i +l} a_i x^{-k_i+j_i-l} = a_i.$$
This shows that $x^{k_i-j_i+l} \in C_{\left<x \right>}(A_i)$ and thus $k_i - j_i + l \equiv 0 \mod r_i$. Furthermore, as $\Phi_{j_1...j_n}|_{\left< x\right>} = \Phi_{k_1 ... k_n}|_{\left< x \right>} = \textup{id}_{\left< x \right>}$,   $$g^{-1}xg = x.$$ Thus, $b \in C_A(\left<x\right>) = Z(G) \cap A$. These calculations show that $\Phi_{j_1...j_n}$ and $\Phi_{k_1...k_n}$ determine the same outer automorphism of $G$ if and only if the differences $k_i - l_i$ are constant modulo $r_i$. As $r_i \mid r_n$ for all $ 1 \leq i \leq n$, because $r_i \leq r_n$ and the $r_i$'s are $p$-powers, it follows that the elements in $K$ induce different elements of $\textup{Out}_{col}(G)$.

It remains to show that every Coleman automorphism induces the same outer automorphism as some element in $K$. Let $\sigma \in \textup{Aut}_{col}(G)$. As $\left< x \right>$ is a $p$-group, there exists an $a \in A$ such that $$ \sigma|_{\left< x \right>} = \textup{conj}(a)|_{\left< x \right>}.$$
Define the automorphism $\widetilde{\sigma}= \sigma \circ \textup{conj}(a^{-1})$. Clearly, $\widetilde{\sigma}$ is a Coleman automorphism of $G$. Moreover, for every $1 \leq i \leq n$, there exists a positive integer $j_i$ such that \begin{align*} \widetilde{\sigma}|_{A_{p_i}} &= \textup{conj}(x^{j_i})|_{A_{p_i}} \end{align*} Hence, $\widetilde{\sigma} = \Phi_{j_1...j_n}$ and $ \tilde{\sigma} \circ \textup{conj}(x^{-j_n}) = \Phi_{k_1,...,k_{n-1},0} \in K$ for some non-negative integers $k_1,...,k_{n-1}$. As $\sigma$ and $\Phi_{k_1,...,k_{n-1},0} $ induce the same element in $\textup{Out}_{col}(G)$, it follows that, $$\textup{Aut}_{col}(G) = \textup{Inn}(G) \rtimes K. \qedhere$$
\end{proof}

An application we can now give an alternative proof of E.C. Dade's result in \cite{DadeLocallytrivial}. The proof makes use of the well known fact that if $p$ is a prime number and $r$ is a non-negative number. Then there exist infinitely many prime numbers $q$ such that there is a primitive $p^r$-th root of unity in $\mathbb{Z}_q$.


\begin{Theorem}
If $H$ is a finite abelian group, then there exists a semi-direct product $A \rtimes B$ of finite finite abelian groups $A$ and $B$ such that $\textup{Out}_{col}(A \rtimes B) \cong H$.
\end{Theorem}
\begin{proof}
First, we deal with the case that $H = \Pi_{i=1}^n C_{p^{r_i}}$, for some prime number $p$ and positive integers $r_1 \leq ... \leq r_n$. Define $r_{n+1} = r_n$ and let $s \geq r_n$ be a non-negative integer. By the above remark, there exist different prime numbers $q_1,...,q_n,q_{n+1}$ and $k_1,...,k_{n+1}$ positive integers, whose images in $\mathbb{Z}_{q_i}$ is a primitive $p^{r_i}$-th root of unity. Write $C_{p^{s}} = \left<x \right>$ and $C_{q_i} = \left< y_i \right>$. Define the semi-direct product $\left( \Pi_{i=1}^{n+1} C_{q_i} \right) \rtimes C_{p^s}$ by $$x^{-1} y_ix = y_i^{k_i}$$
By Theorem \ref{Theoremoutcol}, $\textup{Out}_{col}\left(\left( \Pi_{i=1}^{n+1} C_{q_i} \right) \rtimes C_{p^s}\right) \cong H$.

Second, let $H$ be any finite abelian group. Write $H = H_{p_1} \times ... \times H_{p_n}$, its decomposition into Sylow subgroups. Then there exist abelian groups $A_1,...,A_n$ and positive integers $s_1,...,s_n$ such that $\textup{Out}_{col}(A_i \rtimes C_{p_i^{s_i}}) \cong H_{p_i}$ Denote $G_i = A_i \rtimes C_{p_i^{s_i}}$. Define the group $G=\Pi_{i=1}^{n} G_i$. Clearly, $G \cong \left(A_1 \times ... \times A_{n} \right) \rtimes \left( C_{p^{s_1}} \times ... \times C_{p^{s_n}} \right)$, where the elements of $C_{p_i^{s_i}}$ act trivially on all $A_j$ except on $A_i$, where it acts like in $G_i$. Then, by Theorem \ref{directprodcoleman}, $\textup{Out}_{col}(G) \cong H$.
\end{proof}
We will now tackle the main theorem of this section. As this proof is quite technical, crucial steps are treated in lemmas.

\begin{Lemma}\label{speciallemma} Let $p$ be a prime number. Let $N$ be a normal subgroup of a finite group $G$. Assume $N$ is nilpotent and $G/N = \left< x N \right> \cong C_{p^n}$ for some positive integer $n$ and $p$-element $x$. Suppose $P$ is a Sylow $q$-subgroup of $N$ and suppose $\sigma \in \textup{Aut}_{col}(G)$ is such that $\sigma|_{\left< x \right> } = \textup{id}_{\left< x \right>}$. If $w \in N$ is a $q$-element and $j$ is a positive integer such that $\sigma|_P = \textup{conj}(wx^j)$, then $w$ and $x^{p^n}$ commute. Furthermore, for any integer $k$, $\left\lbrack x^k, w^{-1} \right\rbrack \in C_{G}(N)$. \end{Lemma}
\begin{proof}
Of course, if $q \neq p$, the result follows from the nilpotency of $G$. So assume $q= p$. Then, $x$ is a $p$-element by assumption and thus we have that $x^{p^n} \in P$. Hence, $x^{-j}w^{-1}x^{p^n} w x^{j} = \sigma(x^{p^n}) = x^{p^n}$. This shows that $\left\lbrack w, x^{p^n} \right\rbrack = 1$. To show the second statement, we note that it is sufficient to show that $\left\lbrack w^{-1},x^k \right\rbrack$ centralizes $P$. To do so, let $a,b \in P$. Then, $$ \sigma( a x^{k} b x) = \sigma(a) \sigma(x^{k}) \sigma( b) \sigma( x),$$ and thus $$\sigma( a x^{k} b x^{-k}) \sigma(x^{k+1})= \sigma(a) \sigma(x^{k}) \sigma( b) \sigma( x). $$ Applying the definition of $\sigma$ on $P$ and $\left< x \right>$ seperately,  $$ x^{-j} w^{-1}a x^{k}bx^{-k} wx^jx^{k+1} = x^{-j} w^{-1}  a w x^j x^k x^{-j} w^{-1} b w x^j x. $$ This can be rewritten as $$x^k b x^{-k} w x^{k} = w x^k w^{-1} b w. $$ Which is equivalent to $$ \left\lbrack w^{-1}, x^k \right\rbrack \in C_G(P).$$ This shows that $\left\lbrack x^{-k},w \right\rbrack \in C_G(P)$.
\end{proof}

\begin{Lemma}
Let $p, G, N$ and $x$ be as in the previous lemma. Let $P_1, ..., P_k$ be the Sylow subgroups of $N$. Let $w_i \in P_i$ be such that $\left\lbrack w_i, x^{p^n} \right\rbrack = 1$ and $\left\lbrack w_i^{-1}, x^l \right\rbrack \in Z(P_i)$ for any $1 \leq i \leq k$and positive integer $l$. Let $j_1,...,j_k$ be positive integers. Then there exists $\Phi_{j_1 ... j_k}^{w_1 ... w_k} \in \textup{Aut}_{col}(G)$ defined via $a_i \mapsto x^{-j_i}w_i^{-1} a_i w_i x^{j_i}$ for $a_i \in P_i$ and $x \mapsto x$.
\end{Lemma}
\begin{proof}
First, we check that the proposed mapping is well-defined. To do so, suppose $ax^j  = bx^t $ with $a,b \in N$ and $j,t $ positive integers. Then $x^{t-j} = b^{-1}a \in N \cap \left< x \right> = \left< x^{p^n} \right>$. Thus there exists an integer $l$ such that $b^{-1}a = x^{lp^n}$. If $p$ does not divide $|N|$, then $N \cap \left< x \right> = 1$. Thus, $a = b$ and $j \equiv t \mod p^n$. Thus we only need to consider the case that $p$ divides $|N|$. Without loss of generality, we may assume that $P_k$ is the Sylow $p$-subgroup of $N$. Thus we need to show that $ax^j = (bx^{lp^n})x^j$ and $ b (x^{lp^n}x^j) = b x^t$ have the same image after applying the relations. Let $b= b_1 ... b_k$ be the decomposition of $b$ into its $p_i$-parts. As $w_k$ and $x^{p^n}$ commute, we find that \begin{align*} (bx^{lp^n}) x^j \mapsto &x^{-j_1}w_1^{-1} b_1 w_1 x^{j_1} ... x^{-j_k} w_k^{-1} b_k x^{lp^n} w_k x^{j_k} x^j \\ = &x^{-j_1}w_1^{-1} b_1 w_1 x^{j_1} ... x^{-j_k} w_k^{-1} b_k  w_k x^{j_k} x^{lp^n} x^j. \end{align*}
On the other hand, $$ b (x^{lp^n}x^j) \mapsto x^{-j_1}w_1^{-1} b_1 w_1 x^{j_1} ... x^{-j_k} w_k^{-1} b_k  w_k x^{j_k} x^{lp^n} x^j. $$ This shows that these relations induce a well-defined map $\Phi_{j_1 ... j_k}^{w_1 ... w_k}$ on $G$.

Second, we show that $\Phi_{j_1 ... j_k}^{w_1 ... w_k}$ is a homomorphism. Let $a, b \in N$ with $a= a_1 ... a_k$ and $b= b_1 ... b_k$ their decompositions into its $p_i$-parts. We should show that $$ \Phi_{j_1 ... j_k}^{w_1 ... w_k}(ax^j b x^l) = \Phi_{j_1 ... j_k}^{w_1 ... w_k}(ax^j) \Phi_{j_1 ... j_k}^{w_1 ... w_k}(bx^l).$$ A quick calculation shows that this is equivalent to the equality \begin{align*} &\left(x^{-j_1}w_1^{-1} a_1 x^j b_1 x^{-j} w_1 x^{j_1} ... x^{-j_k} w_k^{-1} a_k x^{j} b_k x^{-j} w_k x^{j_k}\right) x^{j+k} \\ = &\left(x^{-j_1}w_1^{-1} a_1 w_1 x^{j} w_1^{-1}b_1w_1 x^{-j} x^{j_1} ... x^{-j_k}w_k^{-1} a_k w_k x^{j} w_k^{-1} b_k w_k x^{-j} x^{j_k}\right) x^{j+k}.\end{align*} Thus we find that for any $ 1 \leq i \leq k$, we need the following equalities $$ x^j b_i x^{-j}w_i = w_i x^{j}w_i^{-1}b_i w_i x^j.$$ This can be rewritten as $$ \left\lbrack w_i^{-1}, x^{j} \right\rbrack b_i = b_i \left\lbrack w_i^{-1}, x^j \right\rbrack.$$ Thus, by the assumptions of the theorem, $\Phi_{j_1 ... j_k}^{w_1 ... w_k}$ is a homomorphism. Clearly, $\Phi_{j_1 ... j_k}^{w_1 ... w_k}$ is a Coleman automorphism of $G$.
\end{proof}
Using the notation of the previous lemmas, we define the following subgroup of $P_i$; $$D_i = \left\lbrace w_i \in P_i \mid \left\lbrack x^{p^n}, w_i \right\rbrack = 1, \left\lbrack x^t, w_i^{-1} \right\rbrack \in C_{G}(N) \textnormal{ for all positive integers } t \right\rbrace.$$
\begin{Theorem} Let $p$ be a prime number and $G$ a nilpotent-by-($p$-power cyclic) group. Let $N$ be a nilpotent normal subgroup of $G$ such that $G/N \cong C_{p^n}$ for some positive integer $n$. Let $x \in G$ be a $p$-element such that $G/N \cong \left< \overline{x} \right>$. Let $P_1,...,P_k$ be the Sylow subgroups of $N$. Let $r_i$ denote the smallest positive integer such that $\textup{conj}(x^{r_i})|_{P_i} = \textup{conj}(h_i)|_{P_i}$ for some $h_i \in P_i$. Renumber such that $r_1 \leq ... \leq r_k$. Put $T_j$ a transversal of $D_j/ \left(Z(P_j)C_{P_j}(\left<x \right>) \left<h_j \right> \cap D_j\right)$, where $P_j$ is the Sylow $p$-subgroup of $N$ (if it is non-trivial) and put $T_i$ a transversal of $D_i/\left(Z(P_i)C_{P_i}(\left<x \right>)\cap D_i\right)$. Then, $$ \textup{Aut}_{col}(G) = \textup{Inn}(G) \rtimes K, $$ with $$ K = \left\lbrace \Phi_{j_1 ... j_{k-1} 0}^{w_1 ... w_k} \mid 0 \leq j_i < r_i, w_i \in T_i \right\rbrace. $$ Furthermore, $\textup{Out}_{col}(G) \cong K$.
\end{Theorem}
\begin{proof}
Note that $h_j$ is the only non-trivial of the $h_i$. Clearly, $K$ is a subgroup of $\textup{Aut}_{col}(G)$.  We will now show that $\textup{Aut}_{col}(G) = \textup{Inn}(G)K$. Let $\sigma \in \textup{Aut}_{col}(G)$. There exists a $w \in N$ and $k$ a positive integer such that $\sigma$ acts as $\textup{conj}(wx^k)$ on the Sylow $p$-subgroup containing $\left<x \right>$. Denote $\tilde{\sigma} = \sigma \circ \textup{conj}(w^{-1})$, then $\tilde{\sigma}|_{\left<x \right>} = \textup{id}_{\left<x \right>}$. As $\tilde{\sigma}$ is a Coleman automorphism, it follows that there exists $w_k \in P_k$ and a positive integer $l$ such that $\tilde{\sigma}|_{P_k} = \textup{conj}(w_k x^l)$. Put $\sigma_1 = \textup{conj}(x^{-l}) \circ \tilde{\sigma}$. Clearly, $\sigma_1 = \Phi_{j_1 ... j_{k-1} 0}^{w_1 ... w_k}$ for some positive integers $j_1, ..., j_{k-1}$ and $w_1 \in P_1, ..., w_{k-1} \in P_{k-1}$. Note that, by Lemma \ref{speciallemma}, it follows that $w_i \in D_i$.

Second, we show that each element of $K$ is different from the rest.  Let $g \in G$, $w_1,s_1 \in P_1, ..., w_k, s_k \in P_k$ and $j_1,...,j_k, l_1,...,l_n$ positive integers. Suppose that $$\Phi_{j_1...j_k}^{w_1 ... w_k} = \textup{conj}(g) \circ \Phi_{l_1...l_k}^{s_1 ... s_k}.$$
Clearly, $g \in C_{G}(\left< x \right>)$ and for any $a_i \in P_i$ we have $$ g^{-1} x^{-l_i} s_i^{-1} a_i s_i x^{l_i} g = x^{-j_i} w_i^{-1} a_i w_i x^{j_i}.$$
Equivalently, denoting $a_i' = g^{-1}s_i^{-1} a_i s_i g$, $$ x^{j_i - l_i} a_i' x^{l_i - j_i} = w_i^{-1} s_i g a_i' g^{-1}s_i^{-1} w_i. $$
Put $g = ax^f$, where $a \in N$. Its Sylow decomposition is $a = a_1 ... a_k$. Hence, $a \in C_N(\left< x \right>)$. Thus, the previous equation is equivalent with $$ x^{-f+j_i-l_i} a_i' x^{l_i -j_i +f} = x^{-f} w_i^{-1} s_i x^f a a_i' a^{-1} x^{-f} s_i^{-1} w_i x^f.$$
Hence, $-f+j_i-l_i \equiv 0 \mod p^{r_i}$ and for the Sylow $p$-subgroup $P_j$ (if it exists) of $N$ it follows that $x^{l_j - j_j} w_j^{-1} s_j x^f a_j \in C_{G}(P_j)$. As this is a normal subgroup of $G$, it follows that $x^{l_j -j_j +f} a_j^{-1} w_j^{-1}s_j \in C_{G}(P_j)$. Or equivalently, $w_j^{-1}s_j \in Z(P_j)C_{P_j}(\left<x \right>) \left<h_j\right>$. As $x$ is a $p$-element, we find for the other Sylow subgroups $P_i$ that $x^{-f}w_i^{-1} s_i x^f a_i \in Z(P_i)$. As both $Z(P_i)$ and $C_{P_i}(\left<x \right>)$ remain invariant under conjugation by $x$, it follows that this is equivalent to $w_i^{-1} s_i \in Z(P_i) C_{P_i}(\left< x \right>)$. This shows that all elements of $K$ are different even up to an inner automorphism and that for every Coleman automorphism $\sigma$ of $G$, there exist positive integers $f_1,...,f_{k-1}$ and $y_1,...,y_k \in N$ such that $\sigma$ and $\Phi_{f_1...f_{k-1}}^{y_1...y_k}$ induce the same outer automorphism. Together with the first paragraph of this proof, this shows that $\textup{Aut}_{col}(G) = \textup{Inn}(G) K$.  Furthermore, it shows that $\textup{Inn}(G) \cap K = 1$. Thus $\textup{Aut}_{col}(G) = \textup{Inn}(G) \rtimes K$.
\end{proof}

\section{Meta-nilpotent groups with abelian-by-cyclic Sylow 2-subgroup}
In this section a slight generalization of J. Hai and S. Ge's recent paper \cite{Hainilpotentbycyclicabeliansylow2group} is shown using earlier stated results. Note that the concept of Coleman automorphism in their paper includes among others the restriction that it has to be class-preserving as well. For completeness' sake, we include the statement of J. Hai and S. Ge's theorem.
\begin{Theorem}\label{nilpotentbycyclic} Let $G$ be a finite group and let $N$ be a nilpotent normal subgroup of $G$ such that $G/N$ is cyclic. If the Sylow $2$-subgroup of $N$ is abelian, then every class-preserving Coleman automorphism of $G$ of $2$-power order is inner. \end{Theorem}
First, we will show that in this theorem the number $2$ can be replaced any prime number. For this, one can easily adapt the proof of J. Hai and S. Ge to the more general case. However, we first need to formulate the following Corollary to Lemma \ref{lemmaopzn}.
\begin{Corollary}\label{Corollaryabelianopz} Let $G$ be a finite group and let $N$ be a normal subgroup of $G$ for which the quotient group $G/N$ is abelian. Let $\sigma$ be a class-preserving Coleman automorphism of $G$. Then $\sigma$ is an inner automorphism of $G$ if and only if $\sigma|_{N P} = \textup{conj}(g)|_{N P}$ for some $g \in G$ and some Sylow $p$-subgroup $P$ of $G$.
\end{Corollary}
To show the above result, one only needs to modify $\sigma$ as before such that $\sigma|_{N P} = \textup{id}_{N P}$. This can be done by considering $\sigma \circ \textup{conj}(g^{-1})$ and taking a suitable $p'$-power. Moreover, it is clear that $\sigma$ induces the identity on $G/N$, by Theorem \ref{p'groupserfelijk}. The result then follows from Lemma \ref{lemmaopzn}.
\begin{Theorem}\label{generalnilpotentbycyclic}  Let $G$ be a finite group and let $N$ be a nilpotent normal subgroup of $G$ such that $G/N$ is cyclic. Let $p$ be a prime number. If the Sylow $p$-subgroup of $N$ is abelian, then every class-preserving Coleman automorphism of $G$ of $p$-power order is inner. \end{Theorem}
The previous result can be extended to nilpotent-by-nilpotent groups. We show that the general theorem for nilpotent-by-nilpotent groups can be reduced to the general theorem for nilpotent-by-cyclic groups.
\begin{Theorem} Let $G$ be a nilpotent-by-nilpotent group. Let $N$ be a nilpotent normal subgroup of $G$ such that $G/N$ is nilpotent. If the Sylow $p$-subgroup of $N$ is abelian and the Sylow $p$-subgroup of $G/N$ is cyclic, then $G$ has no non-inner class-preserving Coleman automorphisms of order a power of $p$. \end{Theorem}
\begin{proof}
Let $G$ be as in the statement of the Theorem. Hence, there exists a $p$-element $x \in G$ such that $N\left< x \right>/N$ is the Sylow $p$-subgroup of $G/N$. Put $M = N \left<x \right>$. Clearly, $M$ is a normal subgroup of $G$ and $p \nmid |G/M|$. So, by Theorem \ref{Corollaryp'groupserfelijk}, the theorem is shown if $\textup{Out}_{col}(M) \cap \textup{Out}_{c}(M)$ is a $p'$-group. As $M$ satisfies all conditions in Theorem \ref{generalnilpotentbycyclic}, this clearly holds.
\end{proof}
\bibliographystyle{plain}
{\footnotesize \bibliography{normaliserbib2}}

\begin{thebibliography}{10}

\bibitem{Blackburn}
N.~Blackburn.
\newblock Finite groups in which the normal subgroups have nontrivial
  intersection.
\newblock {\em Journal of Algebra}, 3:30 -- 37, 1966.

\bibitem{coleman1963}
D.B. Coleman.
\newblock On the modular group ring of a $p$-group.
\newblock {\em Proc. Amer. Math. Soc.}, 15:511 -- 514, 1963.

\bibitem{DadeLocallytrivial}
E.C. Dade.
\newblock Locally trivial outer automorphisms of finite groups.
\newblock {\em Math.Z.}, 114:57--76, 1970.

\bibitem{DadeCorrection}
E.C. Dade.
\newblock Correction to ''locally trivial outer automorphisms of finite
  groups''.
\newblock {\em Math.Z.}, 124:259--260, 1972.

\bibitem{grosspcentral}
F.~Gross.
\newblock Automorphisms which centralize a {S}ylow $p$-subgroup.
\newblock {\em Journal of Algebra}, 77:202--233, 1982.

\bibitem{Hainilpotentbycyclicabeliansylow2group}
J.~Hai and S.~Ge.
\newblock On {C}oleman automorphisms of finite nilpotent groups by cyclic
  groups.
\newblock {\em Acta Mathematica Sinica, English series}, 32(12):1459--1464,
  2016.

\bibitem{Haiholomorphnilpotentsymmetric}
J.~Hai, S.~Ge, and W.~He.
\newblock The normalizer property for integral group rings of holomorphs of
  finite nilpotent groups and the symmetric groups.
\newblock {\em J. Algebra Appl.}, 16(2), 2017.

\bibitem{haipermutational}
J.~Hai and J.~Guo.
\newblock The normalizer property for the integral group ring of the wreath
  product of two symmetric groups ${S}_k$ and ${S}_n$.
\newblock {\em Comm.Algebra}, 45(3):1278 -- 1283, 2017.

\bibitem{classpreservingHertweck}
M.~Hertweck.
\newblock Class-preserving automorphisms of finite groups.
\newblock {\em Journal of Algebra}, 241(1):1--26, 2001.

\bibitem{Counterexample}
M.~Hertweck.
\newblock A counterexample to the isomorphism problem for integral group rings
  of finite groups.
\newblock {\em Annals of Mathematics}, pages 115 -- 138, 2001.

\bibitem{Hertwecklocalanalysis}
M.~Hertweck.
\newblock Local analysis of the normalizer problem.
\newblock {\em Journal of Pure and applied Algebra}, 163:259 -- 276, 2001.

\bibitem{Hertweckclasspreservingcoleman}
M.~Hertweck.
\newblock Class-preserving coleman automorphisms of finite groups.
\newblock {\em Monatsh. Math.}, 136:1--7, 2002.

\bibitem{blackburnJespers}
M.~Hertweck and E.~Jespers.
\newblock Class-preserving automorphisms and the normalizer property for
  {B}lackburn groups.
\newblock {\em Journal of Group Theory}, 12:157 -- 169, 2009.

\bibitem{HertweckKimmerle}
M.~Hertweck and W.~Kimmerle.
\newblock Coleman automorphisms of finite groups.
\newblock {\em Math.Z.}, 242:203 -- 215, 2001.

\bibitem{Huppert}
B.~Huppert.
\newblock {\em Endliche Gruppen I}.
\newblock Springer, 1967.

\bibitem{Jackowski1987}
S.~Jackowski and Z.~Marciniak.
\newblock Group automorphisms inducing the identity map on cohomology.
\newblock {\em J. Pure Appl. Algebra}, 44:241 -- 250, 1987.

\bibitem{nilpotentbycyclic}
Z.~Li and J.~Hai.
\newblock The normalizer property for integral group rings of wreath products
  of finite nilpotent groups by cyclic groups.
\newblock {\em Communications in Algebra}, 39:521 -- 533, 2011.

\bibitem{generalizeddihedral}
Z.~Li and Y.~Li.
\newblock Coleman automorphisms of generalized dihedral groups.
\newblock {\em Acta Mathematica Sinica}, 32(2):251--257, 2016.

\bibitem{Mazur1995}
M.~Mazur.
\newblock On the isomorphism problem for integral group rings of infinite
  groups.
\newblock {\em Expo.Math.}, 13:433--445, 1995.

\bibitem{Mazurinfinitegroups}
M.~Mazur.
\newblock The normalizer of a group in the unit group of its group ring.
\newblock {\em Journal of Algebra}, 212:175 -- 189, 1999.

\bibitem{saksonov}
A.I. Saksonov.
\newblock Group rings of finite groups.i.
\newblock {\em Publ.Math.Debrecen}, 18:187--209, 1971.

\bibitem{Sehgalbook}
S.K. Sehgal.
\newblock {\em Units in integral group rings}.
\newblock Longman, 1993.

\end{thebibliography}
\end{document}